\documentclass{gtpart}     
\usepackage{latexsym}
\usepackage{amsfonts}
\usepackage{amsbsy}
\usepackage{mathrsfs}
\usepackage{epsfig}
\usepackage{breakurl}
\usepackage[all]{xy}
\usepackage{url}
\usepackage{tikz}
\usetikzlibrary{arrows,matrix}
\usetikzlibrary[calc]
\usetikzlibrary[patterns]

\usepackage{hyperref}
\usepackage{relsize}
\def\Cat{\mathcal{C}\!at}
\title{Some Remarks on Realization of Simplical Algebras in $\Cat$}

\author[Fiedorowicz]{Z. Fiedorowicz}
\givenname{Zbigniew}
\surname{Fiedorowicz}
\address{Department of Mathematics, The Ohio State University\\ Columbus, OH 43210-1174, USA}
\email{fiedorow@math.ohio-state.edu}
\urladdr{http://www.math.ohio-state.edu/people/fiedorow/view}

\author[Vogt]{R.M.~Vogt}
\givenname{Rainer}
\surname{Vogt}
\address{Universit\"at Osnabr\"uck, Fachbereich Mathematik/Informatik\\ Albrechtstr. 28a, 49069 Osnabr\"uck, Germany}
\email{rainer@mathematik.uni-osnabrueck.de}
\urladdr{http://www.mathematik.uni-osnabrueck.de/staff/phpages/vogtr.rdf.shtml}

\keyword{Categories, Operads}
\subject{primary}{msc2000}{18D50}
\subject{secondary}{msc2000}{55P48}

\volumenumber{}
\issuenumber{}
\publicationyear{}
\papernumber{}
\startpage{}
\endpage{}
\doi{}
\MR{}
\Zbl{}
\received{}
\revised{}
\accepted{}
\published{}
\publishedonline{}
\proposed{}
\seconded{}
\corresponding{}
\editor{}
\version{}

\makeatletter
\@addtoreset{equation}{section}

\makeatother
\makeatletter
\@addtoreset{equation}{section}

\makeatother

\newtheorem{prop}{Proposition}[section]
\newtheorem{theo}[prop]{Theorem}

\newtheorem{const}[prop]{Constructions}
\newtheorem{coro}[prop]{Corollary}

\theoremstyle{definition}
\newtheorem{defi}[prop]{Definition}

\newtheorem{exam}[prop]{Example}
\newtheorem{leer}[prop]{}
\newtheorem{rema}[prop]{Remark}

\def\scO{\mathcal{O}}
\def\hscO{\widehat{\mathcal{O}}}

\def\scC{\mathcal{C}}
\def\scD{\mathcal{D}}

\def\scT{\mathcal{T}}

\def\scA{\mathcal{A}}
\def\scB{\mathcal{B}}
\def\scD{\mathcal{D}}

\def\scF{\mathcal{F}}
\def\scL{\mathcal{L}}
\def\scK{\mathcal{K}}
\def\scS{\mathcal{S}}

\def\tmT{\widetilde{\mathbb{T}}}

\def\colim{\textrm{colim}}
\def\Id{{\textrm{Id}}}
\def\id{{\textrm{id}}}
\def\op{{\textrm op}}
\def\ob{{\textrm ob}}
\def\Top{\mathcal{T}\!op}
\def\STop{\mathcal{ST}\!op}
\def\SSets{\mathcal{SS}ets}
\def\SSSets{\mathcal{S}^2\mathcal{S}ets}
\def\SCat{\mathcal{SC}at}
\def\Sets{\mathcal{S}\!ets}

\def\In{\textrm{In}}
\def\End{\textrm{End}}

\def\cat{\textrm{cat}}

\def\cat{\mbox{cat}}
\def\sd{\mbox{sd}}
\def\esd{\mbox{esd}}
\def\ssd{\mbox{ssd}\,}
\def\diag{\mbox{diag}\,}
\begin{document}
\begin{abstract}    
In this paper we discuss why the passage from simplicial algebras over a $\Cat$ operad to algebras over that operad
involves apparently unavoidable technicalities.
\end{abstract}

\maketitle

\vspace{2ex}
\section{Introduction}

One direction of research in homotopy theory has involved comparing algebraic structures in $\Top$ to corresponding structures
in $\Cat$.  For instance in algebraic K-theory one often starts with some algebraic structure on a category and converts it into
a corresponding structure on a topological space or spectrum.  A natural question is to what extent can this approach be reversed,
 i.e.  to what extent do algebraic structures in $\Top$ correspond to algebraic structures in  $\Cat$?

Thomason \cite{Thom2} was the first to consider this question.  He showed that symmetric monoidal categories model all connective
spectra.  In a series of papers \cite{FV}, \cite{FSV}, \cite{FSV2}, we considered the same question for iterated loop spaces and we
showed that iterated monoidal categories model all such spaces.

In both Thomason's work and ours, the most technical part of the proof involves the passage from simplicial $\Cat$-algebras over a
$\Cat$-operad to plain $\Cat$-algebras over that operad, a process we refer to as \textit{rectification}.  It has been suggested to us that we
might avoid these technicalities if we construct an appropriate categoric realization functor.  In this paper we will discuss why we believe
that such a simple realization construction is not possible.

\section{Notations and definitions}\label{sec.not}

First let us clarify what we mean by an ``appropriate'' categoric realization functor. For this we have to introduce some notation.

\begin{leer}\label{2_1} \textbf{Notations:}
 \begin{enumerate}
  \item $\Delta$ denotes the category of posets $[n]=\{0<1<2<\dots<n\}$ and order preserving maps. We usually write
$\Delta(k,n)$ rather than $\Delta([k],[n])$ for the morphism sets.
  \item $\Cat,\ \Sets, \ \Top$ denote the categories of small categories, of sets, and of $k$-spaces respectively.
  \item $\SCat,\ \SSets,\ \STop$ denote the associated categories of simplicial objects.
  \item $\SSSets$ denotes the category of bisimplicial sets.
  \item For a functor $F:\scC\to \scD$ and an small category $\scK$ let $F^{\scK}: \scC^{\scK}\to \scD^{\scK}$ denote its prolongation
  to the functor categories.
  \item A natural transformation $\alpha:F\Rightarrow G$ of functors $F,\ G:\scC\to \Top$ is called a {\em weak equivalence}, if the maps
  $\alpha(C):F(C)\to G(C)$ are homotopy equivalences.
  \item Two such functors are called {\em equivalent} if there is a chain of weak equivalences connecting them.
  \item $B:\Cat\to \Top$ denotes the classifying space functor, i.e. the composite of the nerve functor $N:\Cat\to \SSets$ 
   and the geometric realization functor $|-|:\SSets\to \Top$.
   \item A functor $F:\scC\to \scD$ between small categories is called a weak equivalence if $B(F): B(\scC)\to B(\scD)$ is a 
   homotopy equivalence.
 \end{enumerate}
\end{leer}

\begin{defi}\label{2_2} A  \textit{categoric realization functor} is a functor $F_D:\SCat\to\Cat$ given by a coend construction of the form
\[F_D(\scC_*) = \scC_*\otimes_{\Delta}D(*)=\left(\amalg_{n\ge0}\scC_n\times D(n)\right)/\approx,\]
where $D:\Delta\to\Cat$ is a fixed cosimplicial category.
\end{defi}

If we take $D_0:\Delta\to\Cat$ to be the constant cosimplicial category on the trivial category $\ast$, we obtain 
$$F_{D_0}(\scC_*)=\colim_{\Delta^{\op}}\scC_\ast$$
which apparently is not what we are looking for. We need a categoric realization functor which has the ``correct'' homotopy type.
Moreover, since we want to replace a simplicial algebra over an operad in $\Cat$ by an algebra in $\Cat$ of the same homotopy type
the realization functor has to be product preserving.

\begin{defi}\label{2_3}
A categoric realization functor $F_D$ is called {\em good} if the functor $B\circ F_D: \SCat\to\Cat\to \Top$ is equivalent to the functor
$|-|\circ B^{\Delta^{\op}}:\SCat\to \STop\to \Top$ or, equivalently, to the functor $|-|\circ \mbox{diag} \circ N^{\Delta^{\op}}:
\SCat\to \SSSets\to \SSets\to \Top$ where $\mbox{diag}$ is the diagonal functor. We call $F_D$
product preserving, if
the natural map $F(\scC_*\times \scD_*)\longrightarrow F(\scC_*)\times F(\scD_*)$ is an isomorphism.
\end{defi}

If $F_D$ is a product preserving categoric realization functor and $\scC_*$ is a simplicial category, then $F_D$ induces an
operad map $\End(\scC_*)\to \End(F_D(\scC_*))$ of endomorphism operads. Hence, if $\scO$ is a $\Cat$-operad and $\scC_*$ is a 
simplicial $\scO$-algebra, then $F_D(\scC_*)$ is an $\scO$-algebra in $\Cat$. If $F_D$ is also good, then the classifying space of this
$\scO$-algebra is equivalent to the classifying space of the original simplicial $\scO$-algebra.
 Thus $F_D$ would provide the desired rectification functor.

\begin{rema}\label{2_4} In what follows we will need to consider iterated coend constructions over $\Delta$.  In such circumstances it
is clearer to use the following notation for coends
\[A_*\otimes_\Delta B^*= A_n\otimes_{n\in\Delta}B^n,\]
where $A_*$, respectively $B^*$, are simplicial, respectively cosimplicial, objects in $\Sets$, $\SSets$, or $\Cat$. This conforms
to the notation for coends in \cite{Mac}.
\end{rema}

The goodness condition on $F_D$ is an indication that $BD(n)$ should be closely related to the standard $n$-simplex $\Delta(n)$.
We start with three obvious candidates, for which $BD(n)\cong \Delta^n$, and explain why each falls short of the mark.  We also discuss a fourth variant, with $BD(n)\simeq\Delta^n$, which also fails.  In the last two sections
we investigate another possible option: replacing the original simplicial category $\scC_*$ by some kind of cofibrant
resolution prior to applying a categoric realization functor.  If we do this, then we obtain a good categoric realization
for any choice of $D$ with $BD(n)\simeq\Delta^n$.  However this results in the loss of the $\scO$-algebra
structure.  This can be remedied by applying the rectification process of  \cite{FSV2} degreewise.  However
this does not result in any simplification of the rectification process.

\section{Standard categoric realization}\label{sec.stan}

The most obvious candidate for a categoric realization functor is the one where we take $D(n)=D_1(n)=[n]$.
We will denote the resulting categoric realization $F_{D_1}$ by $F_1$.

Before we explain why this fails to be a good categoric realization functor, let us start with an elementary observation.  Since we
can think of a set as being the same thing as a discrete category, we can regard a simplicial set as being a special case of a
simplicial category.  Thus any categoric realization functor restricts to a functor ${\SSets\longrightarrow\Cat}$. Since these functors
are constructed as coends, they preserve colimits.  Such colimit-preserving functors ${\SSets\longrightarrow\Cat}$ are generically
referred to as \textit{categorification functors}.

There is a standard categorification functor ${\cat:\SSets\longrightarrow\Cat}$.  This can be briefly described as
the left adjoint to the nerve functor ${N:\Cat\longrightarrow\SSets}$.  A more explicit description is as follows.  Given a simplicial set
$S_*$ one associates to it the directed graph whose vertices are the 0-simplices $S_0$.  The edges are the 1-simplices $S_1$.  Each
edge $x$ is directed from $d_1x$ to $d_0x$.  Then one takes the free category on this directed graph.  Finally for each 2-simplex
$y\in S_2$, one identifies the composite $(d_0 y)(d_2y)$ with $d_1(y)$. The resulting quotient category is called $\cat(S_*)$.

\begin{leer}\label{3_1}
The standard categorification functor ${\cat:\SSets\longrightarrow\Cat}$ has the following nice properties:
\begin{itemize}
\item[{\rm (i)}] $\cat\circ N= Id$
\item[{\rm (ii)}] $\cat$ preserves products, i.e. the natural map $\cat(S_*\times T_*)\longrightarrow\cat(S_*)\times\cat(T_*)$
is an isomorphism.
\end{itemize}
\end{leer}

The first property is an immediate consequence of the definition.  A proof of the second property may be found in \cite[p. 1097]{FSV}.

\begin{prop}\label{3_2} The restriction to simplicial sets of the categoric realization functor $F_1$ corresponding to the cosimplicial category $D_1(n)=[n]$
is the standard categorification functor ${\cat:\SSets\longrightarrow\Cat}$. Moreover, $F_1$ is product preserving, and for any simplicial category $\scC_*$ we have
$F_1(\scC_*)=\cat\left(\diag\,N^{\Delta^{\op}}\scC_*\right)$.
\end{prop}

\begin{proof} Let $S_*$ be a simplicial set.  Then we have
\[S_k=S_*\otimes_{\Delta}\Delta^*_k=\left(\amalg_{n\ge0}S_n\times\Delta(k,n)\right)/\approx\]
where $\Delta(*,n)=\Delta^n_*$ is the standard simplicial set model of $\Delta^n$,
 and the equivalence relation is given by the standard face and degeneracy relations in $\SSets$.  If we regard the set $S_n$ as a discrete simplicial set,
 we have the following sequence of equalities
\begin{eqnarray*}
\cat(S_*) &= &\left(\amalg_{n\ge0}\cat(S_n\times\Delta(*,n))\right)/\approx\\
&= &\left(\amalg_{n\ge0}\cat(S_n)\times\cat(\Delta(*,n)\right)/\approx\\
&= &\left(\amalg_{n\ge0}S_n\times[n]\right)/\approx\\
&= &S_*\otimes_\Delta D_1(*)\\
&= &F_1(S_*)
\end{eqnarray*}
Here the first equality is due to the fact that $\cat$ preserves colimits.  The second equality follows from the fact that $\cat$ is product
preserving \ref{3_1}.(ii).  The third equality follows from $\cat(S_n)=S_n$, since $S_n$ is discrete, and 
$\cat(\Delta(-,n))=\cat(\Delta^n_*)=(\cat\circ N)[n]=[n]$, c.f. \ref{3_1}.(i). The fourth equality is just the definition of coend and the last
equality is the definition of $F_1$.

For the second statement, we note that
\[\diag\,N^{\Delta^{\op}}\scC_*=N_*\scC_n\otimes_{n\in\Delta}\Delta(*,n)=N_*\scC_n\otimes_{n\in\Delta}N_*[n].\]
Here we use the notation $N_*$ instead of $N$ for nerve, in order to emphasize that the nerve functor takes values in $\SSets$.
Specifically, the set of $m$-simplices of the right hand side of the above equality is the following coend in $\Sets$:
\[N_m\scC_n\otimes_{n\in\Delta}\Delta(m,n)=N_m\scC_n\otimes_{n\in\Delta}N_m[n].\]
Since $\cat$ preserves colimits and products, it follows that
\begin{eqnarray*}
\cat\left(\mbox{diag}\,N^{\Delta^{\op}}\scC_*\right) &= &\cat(N_*\scC_n)\otimes_{n\in\Delta}\cat(N_*[n])\\
&= &\scC_n\otimes_{n\in\Delta}[n]\\
&= &\scC_*\otimes_\Delta D_1(*)\\
&= &F_1(\scC)
\end{eqnarray*}
Here the second equality follows from $\cat\circ N=Id$ and the rest follows from definition.

Since $\cat, \ \mbox{diag}$, and $N^{\Delta^{\op}}$ preserve products, so does $F_1$.
\end{proof}

\begin{coro}$F_1$ is product preserving, but not good.
\end{coro}

\begin{proof}
By Proposition \ref{3_2}, the functor $F_1$ is product preserving. However Proposition \ref{3_2} also rules out $F_1$ as good, because this 
would require
that $F_1$ restricted to simplicial sets should be a homotopy inverse to the nerve functor.  It is well known that $\cat$ is not such a
 functor, since $\cat$ is by definition completely determined by its restriction to the 2-skeleton of a simplicial set.
\end{proof}

\begin{exam}\label{3_4}
A counterexample is $\Delta^n_*/\partial\Delta^n_*$, with $n\ge2$.  Then ${N\cat\left(\Delta^n_*/\partial\Delta^n_*\right)=*}$,
whereas $\Delta^n_*/\partial\Delta^n_*$ has the homotopy type of the $n$-sphere. This is closely related to the well known fact
that the nerve functor does not preserve the homotpy type of pushouts in $\Cat$.  A simple example of this is the following pushout
diagram in $\Cat$
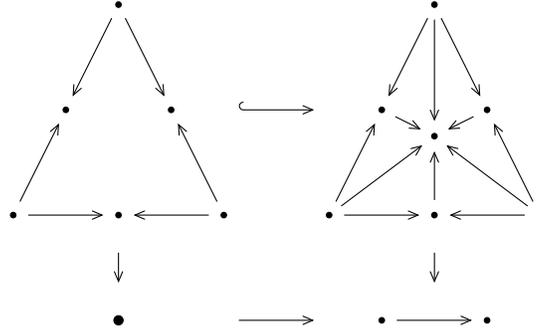
$$
 \begin{tikzpicture}[>=angle 45,shorten >=0.2cm,shorten <=0.2cm,->,scale=0.35]
  \draw[->](0,0) -- (4,0);
	\draw[->](0,0) -- (2,4);
	\draw[<-](2,4) -- (4,8);
	\draw[->](8,0) -- (4,0);  
	\draw[->](8,0) -- (6,4); 
	\draw[<-](6,4) -- (4,8); 
	\draw[fill] (0,0) circle(3pt); 
	\draw[fill] (4,0) circle(3pt);
	\draw[fill] (8,0) circle(3pt);
	\draw[fill] (2,4) circle(3pt); 
	\draw[fill] (6,4) circle(3pt);
	\draw[fill] (4,8) circle(3pt);
  \path[right hook->](8,4)edge (12,4);
  \draw[->](12,0) -- (16,0);
	\draw[->](12,0) -- (14,4);
	\draw[<-](14,4) -- (16,8);
	\draw[->](20,0) -- (16,0);  
	\draw[->](20,0) -- (18,4); 
	\draw[<-](18,4) -- (16,8); 
	\draw[->](12,0) -- (16,3);
	\draw[->](14,4) -- (16,3);
	\draw[->](16,8) -- (16,3);
	\draw[->](18,4) -- (16,3);
	\draw[->](20,0) -- (16,3);
	\draw[->](16,0) -- (16,3);
	\draw[fill] (12,0) circle(3pt); 
	\draw[fill] (16,0) circle(3pt);
	\draw[fill] (20,0) circle(3pt);
	\draw[fill] (14,4) circle(3pt); 
	\draw[fill] (18,4) circle(3pt);
	\draw[fill] (16,8) circle(3pt);
	\draw[fill] (16,3) circle(3pt);
	\draw[shorten >=0.5cm,shorten <=0.5cm,->](4,0) -- (4,-4);
	\draw[shorten >=0.5cm,shorten <=0.5cm,->](16,0) -- (16,-4);
	\draw[->](8,-4) -- (12,-4);
	\draw[->](14,-4) -- (18,-4);
	\draw[fill] (4,-4) circle(5pt);
	\draw[fill] (14,-4) circle(3pt); 
	\draw[fill] (18,-4) circle(3pt);
 \end{tikzpicture} 
$$

Thus the nerve of the pushout in $\Cat$ is contractible.  On the other hand the pushout of the correponding diagram of nerves
has the homotopy type of $S^2$.
\end{exam}

\section{Double barycentric subdivision}\label{sec.bary}

The work of Fritsch, Latch, Thomason and Wilson \cite{FL}, \cite{LTW}, \cite{T2} shows that $\cat\circ\sd^2$ is a homotopy 
inverse to the nerve functor, where $\sd$ denotes barycentric subdivision and $\sd^2$ the double barycentric
subdivision.  This suggests that the cosimplicial category
$D_2(n)=\cat\circ\sd^2\left(\Delta^n_*\right)$ might give a good categoric realization.

First let us recall that the barycentric subdivision of a simplicial set $S_*$ can be described as follows:
\[\sd(S_*)=S_n\otimes_{n\in\Delta}N\scF_n=S_n\otimes_{n\in\Delta}\sd(\Delta^n_*),\]
where $\scF_n$ is the poset of faces of $\Delta^n$. 

\begin{prop}\label{4_1} For any simplicial category $\scC_*$ we have
\[F_2(\scC_*)=F_{D_2}(\scC_*)=\cat\left(\sd^2\diag N^{\Delta^{\op}}\scC_*\right).\]
Thus $F_2$ is a good categoric realization.
\end{prop}

\begin{proof} The proof is similar to that of the second part of Proposition \ref{3_2}. First of all we have
\begin{eqnarray*}
\sd^2\diag N^{\Delta^{\op}}\scC_* &= &\sd^2\left(N_*\scC_n\otimes_{n\in\Delta}\Delta(*,n)\right)\\
&= &\left(N_*\scC_n\otimes_{n\in\Delta}\Delta(m,n)\right)\otimes_{m\in\Delta}\sd^2(\Delta(*,n))\\
&= &N_*\scC_n\otimes_{n\in\Delta}\left(\Delta(m,n)\otimes_{\Delta}\sd^2(\Delta(*,m))\right)\\
&= &N_*\scC_n\otimes_{n\in\Delta}\sd^2(\Delta(*,n))
\end{eqnarray*}

Since $\cat$ preserves products and colimits, this implies that
\begin{eqnarray*}
\cat\left(\sd^2\diag N^{\Delta^{\op}}\scC_*\right) &= &\cat\left(N_*\scC_n\otimes_{n\in\Delta}\sd^2(\Delta(*,n))\right)\\
&= &\cat\circ N_*(\scC_n)\otimes_{n\in\Delta}\cat\left(\sd^2(\Delta^n_*)\right)\\
&= &\scC_n\otimes_{n\in\Delta} D_2(n)\\
&= &F_2(\scC_*)\\
\end{eqnarray*}
\end{proof}

Unfortunately $F_2$ is not product preserving,
since barycentric subdivision is not product preserving.  For instance the barycentric subdivision
 $\mbox{sd}\left(\Delta^1_*\times\Delta^1_*\right)$ is given by the following picture
 $$
 
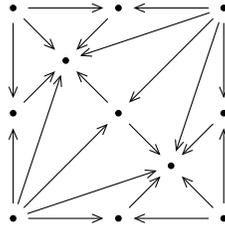
\begin{tikzpicture}[>=angle 45,shorten >=0.2cm,shorten <=0.2cm,->,scale=0.35]
  \draw[->](0,0) -- (4,0);
	\draw[->](0,0) -- (0,4);
	\draw[->](0,0) -- (4,4);
	\draw[->](0,0) -- (2,6);
	\draw[->](0,0) -- (6,2);
	\draw[->](0,4) -- (2,6);
	\draw[->](0,8) -- (4,8);
	\draw[->](0,8) -- (0,4);
	\draw[->](0,8) -- (2,6);
	\draw[->](4,8) -- (2,6);
	\draw[->](4,0) -- (6,2);
	\draw[->](8,0) -- (4,0);  
	\draw[->](8,0) -- (8,4);
	\draw[->](8,0) -- (6,2);
	\draw[->](8,4) -- (6,2);		
	\draw[->](8,8) -- (8,4);
	\draw[->](8,8) -- (4,4);
	\draw[->](8,8) -- (4,8);
	\draw[->](8,8) -- (2,6);
	\draw[->](8,8) -- (6,2);	
	\draw[->](4,4) -- (6,2);
	\draw[->](4,4) -- (2,6);
	\draw[fill] (0,0) circle(3pt); 
	\draw[fill] (0,4) circle(3pt);	
	\draw[fill] (0,8) circle(3pt); 
	\draw[fill] (4,0) circle(3pt);
	\draw[fill] (4,4) circle(3pt); 
	\draw[fill] (4,8) circle(3pt);
	\draw[fill] (8,0) circle(3pt);
	\draw[fill] (8,4) circle(3pt);
	\draw[fill] (8,8) circle(3pt);	
	\draw[fill] (6,2) circle(3pt);
	\draw[fill] (2,6) circle(3pt);
 \end{tikzpicture} 
$$

On the other hand  $\mbox{sd}\left(\Delta^1_*\right)\times\mbox{sd}\left(\Delta^1_*\right)$ is given by the following picture
$$
 
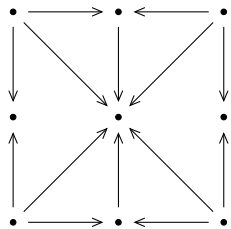
\begin{tikzpicture}[>=angle 45,shorten >=0.2cm,shorten <=0.2cm,->,scale=0.35]
  \draw[->](0,0) -- (4,0);
	\draw[->](0,0) -- (0,4);
	\draw[->](0,0) -- (4,4);
	\draw[->](0,8) -- (4,8);
	\draw[->](0,8) -- (0,4);
	\draw[->](0,8) -- (4,4);
	\draw[->](4,0) -- (4,4);
	\draw[->](4,8) -- (4,4);		
	\draw[->](8,0) -- (4,0);  
	\draw[->](8,0) -- (8,4); 
	\draw[->](8,0) -- (4,4); 	
	\draw[->](8,8) -- (8,4);
	\draw[->](8,8) -- (4,4);
	\draw[->](8,8) -- (4,8);	
	\draw[fill] (0,0) circle(3pt); 
	\draw[fill] (0,4) circle(3pt);	
	\draw[fill] (0,8) circle(3pt); 
	\draw[fill] (4,0) circle(3pt);
	\draw[fill] (4,4) circle(3pt); 
	\draw[fill] (4,8) circle(3pt);
	\draw[fill] (8,0) circle(3pt);
	\draw[fill] (8,4) circle(3pt);
	\draw[fill] (8,8) circle(3pt);	
 \end{tikzpicture} 
$$

\section{Iterated edgewise subdivisions}\label{sec.edge}

As we noted above, one of the major defficiencies of barycentric subdivision is that is not product preserving.  There are two
subdivision constructions which are product preserving.  They are based on subdivision of the edges of a simplicial set.  Both
constructions use the monoidal structure of the category $\Delta$, given by taking the disjoint union of totally ordered finite
sets.  This defines a functor $<\!\!\!\!+\!\!\!\!>:\Delta\times\Delta\longrightarrow\Delta$ and hence also a functor
$<\!\!\!\!+\!\!\!\!>:\Delta^{op}\times\Delta^{op}\longrightarrow\Delta^{op}$.  Given a simplicial set $S_*$, we define
its \textit{edgewise subdivision} $\mbox{esd}(S_*)$ to be the simplicial set
\[\Delta^{op}\stackrel{(id,id)}{\longrightarrow}\Delta^{op}\times\Delta^{op}\stackrel{<\!\!+\!\!>}{\longrightarrow}\Delta^{op}
\stackrel{S_*}{\longrightarrow}\Sets.\]
Segal \cite{Se} constructed a variant of this subdivision, which has certain advantages.  This is based on the functor $r:\Delta\to\Delta$
which reverses the order of a totally ordered set.  Given a simplicial set $S_*$, we define its \textit{Segal subdivision} to be the
simplicial set $\mbox{ssd}(S_*)$
\[\Delta^{op}\stackrel{(r,id)}{\longrightarrow}\Delta^{op}\times\Delta^{op}\stackrel{<\!\!+\!\!>}{\longrightarrow}\Delta^{op}
\stackrel{S_*}{\longrightarrow}\Sets.\]
Explicitly we have
\[\mbox{esd}(S_*)_n = \mbox{ssd}(S_*)_n=S_{2n+1}.\]
The elementary faces and degeneracies for $\mbox{esd}(S_*)$ are given by
\[d_i^{\mbox{\tiny esd}(S_*)}=d_id_{i+n+1},\qquad s_i^{\mbox{\tiny esd}(S_*)}=s_{i+n+1}s_i,\qquad i=0,1,2,\dots, n\]
The elementary faces and degeneracies for $\mbox{ssd}(S_*)$ are given by
\[d_i^{\mbox{\tiny ssd}(S_*)}=d_{n-i}d_{n+1+i},\qquad s_i^{\mbox{\tiny ssd}(S_*)}=s_{n+1+i}s_{n-i},\qquad i=0,1,2,\dots, n\]
It is clear from these definitions that both edgewise subdivisions preserve products.

The following pictures illustrate these subdivisions for the standard 2-simplex $\Delta^2_*$.  Then $\mbox{esd}(\Delta^2_*)$ is
represented by
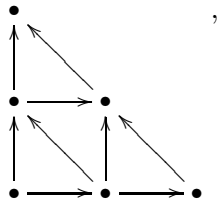
\[\begin{xymatrix}{
\bullet \\
\bullet\ar[u]\ar[r] &\bullet\ar[lu] \\
\bullet\ar[u]\ar[r]  &\bullet\ar[lu]\ar[u]\ar[r] &\bullet\ar[lu]
}\end{xymatrix},\]
whereas $\mbox{ssd}(\Delta^2_*)$ is represented by
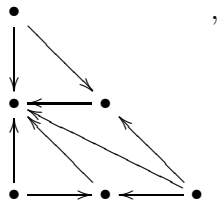
\[\begin{xymatrix}{
\bullet\ar[d]\ar[rd] \\
\bullet&\bullet\ar[l] \\
\bullet\ar[u]\ar[r]  &\bullet\ar[lu]&\bullet\ar[lu]\ar[llu]\ar[l]
}\end{xymatrix},\]
It is clear from the picture above that the edgewise subdivision of a simplex is not the nerve of a category, since it is not closed under
composition of arrows.  If we apply the
categorification functor to the edgewise subdvision of $\Delta^n_*$ we obtain a category whose classifying space is 
$2n$-dimensional.  

On the other hand $D(n)=D_3^{(k)}(n)=\cat\circ\ssd^k(\Delta_*^n)$ does provide a categorification functor $F_3^{(k)}$ with $BD(n)\cong\Delta^n$ for any value $k$.
This follows from the fact that the Segal edgewise subdivision preserves nerves of categories.
 For if $S_*$ is the nerve of a category $\scC$, then $\ssd(S_*)$
is the nerve of the category $\scC'$ whose objects are the morphisms $C\longrightarrow D$ of $\scC$ and whose morphisms are
commutative diagrams
\[\begin{xymatrix}{
C_1\ar[r] &D_1\ar[d]\\
C_2\ar[u]\ar[r] &D_2
}\end{xymatrix}\]

Since the Segal edgewise subdivision is product preserving, $D_3^{(k)}$, for some fixed value of $k$, might have a chance to provide a very good categoric realization functor. 
However, it is not good for much the same reason as $D_1$.  For by the same argument as
in the proof of Proposition \ref{3_2} or \ref{4_1}, this categoric realization functor would take a simplicial set $S_*$  (regarded as a simplicial
category) to $\cat\circ\ssd^k(S_*)$.  Thus the goodness condition for simplicial sets requires this construction to be a homotopy inverse to the
nerve functor.  However a minor variation of Example \ref{3_4} shows that this is not the case.  For any vertex in $\ssd^k(\Delta^n_*)$ has at most $2^k$ nonzero 
barycentric coordinates in $\Delta^n$. Thus if $n=2^k-1$,
then there is precisely one vertex in the interior of $\Delta^n$, namely the barycenter of the simplex, and this vertex is a terminal
object in $\cat\circ\ssd^k(\Delta^n_*)$.  It follows that the following is a pushout diagram in $\Cat$:
\[\xymatrix{
\cat\circ\ssd^k(\partial\Delta^n_*)\ar[d]&\ar@{^{(}->}[r] &&\qquad\cat\circ\ssd^k(\Delta^n_*)\ar[d]\\
\mathlarger{\bullet}&\ar[r] &&\bullet\longrightarrow\bullet
}\]
Thus the nerve of the pushout in $\Cat$ is contractible, whereas the pushout of the nerves has the homotopy type
of $S^n$.

For much the same reason $D(n)=D_4^{(k)}(n)=\cat\circ\esd^k(\Delta_*^n)$ does not give a good
categoric realization functor for any fixed value of $k$. Again this would require that the restriction of
this functor to simplicial sets be a homotopy inverse to the nerve functor.  For by the same reason as
in the Segal edgewise subdivision, if $n\ge 2^k$, then every object in $\cat\circ\esd^k(\partial\Delta_*^n)$ is 
also an object of $\cat\circ\esd^k(\partial\Delta_*^n)$.  Since $\cat\circ\esd^k(\Delta_*^n)$ and
$\cat\circ\esd^k(\partial\Delta_*^n)$ both contain a common terminal object, it follows that the following is a
pushout diagram in $\Cat$:
\[\xymatrix{
\cat\circ\esd^k(\partial\Delta^n_*)\ar[d]&\ar@{^{(}->}[r] &&\qquad\cat\circ\esd^k(\Delta^n_*)\ar[d]\\
\mathlarger{\bullet}&\ar[r] &&\mbox{\ }\mathlarger{\bullet}
},\]
whereas the pushout of the nerves has the homotopy type of $S^n$.

We summarize

\begin{prop} The categoric realization functors $F_3$ and $F_4$ are product preserving, but they are
not good.
\end{prop}

\section{Resolutions}\label{sec.res}

Section \ref{sec.bary} indicates that some type of resolutions might help: $D_2$ is a degreewise cofibrant replacement of $D_1$, if we give $\Cat$ 
the model category structure of Thomason \cite{T2}. It is well known that coends behave rather badly with respect to homotopy. In $\Top$ one therefore
replaces them by the 2-sided bar construction. There is a related construction in $\Cat$ which has been studied by Heggie \cite{Heg} and others.

Let $\scK$ be a small category, and $F:\scK^{\op}\to \Cat$ and $G:\scK\to \Cat$ be functors. Define a category $C(F,\scK,G)$ as follows: objects are
triples $(x,k,y)$ with $k\in \ob\scK,\ x\in \ob F(k),\ y\in\ob G(k)$. A morphism 
$$((x_0,k_0,y_0)\to (x_1,k_1,y_1)$$
is a triple $(f,\alpha,g)$ consisting of a morphism $\alpha:k_0\to k_1$ in $\scK$, a morphism $f:x_0\to F(\alpha)(x_1)$ in $F(k_0)$, and a morphism
$g:G(\alpha)(y_0)\to y_1$ in $G(k_1)$. Composition is defined by
$$ (f_2,\alpha_2,g_2)\circ (f_1,\alpha_1,g_1)=(F(\alpha_1)(f_2)\circ f_1,\alpha_2\circ\alpha_1,g_2\circ G(\alpha_2)(g_1)).$$
This construction is functorial in the obvious sense. 

\begin{prop}\label{6_1} \cite[Thm. 2.5]{Heg} If $\beta: F\Rightarrow F'$ and $\gamma: G\Rightarrow G'$ are natural transformations such that $\beta(k)$ and
 $\gamma(k)$ are weak equivalences for all objects $k\in \scK$, then the induced map
 $$ C(F,\scK,G)\to C(F',\scK,G')$$
 is a weak equivalence.
\end{prop}

\begin{leer}\label{6_2}
 \textbf{Properties:} Let $\ast$ denote the constant diagram on the trivial category $\ast$. Let $\underline{\scK}$ denote the functor
 $$\underline{\scK}: \scK^{\op}\times \scK\to \Cat, \quad (k_0,k_1)\mapsto \scK(k_0,k_1),$$
 where the set $\scK(k_0,k_1)$ is regarded as a discrete category.
 \begin{enumerate}
  \item $C(\ast,\scK,G)$ is the Grothendieck construction $\scK\int G$ studied in \cite{T}, and $C(F,\scK,\ast)= F\int \scK$, the
  dual Grothendieck construction.
  \item $C(\ast,\scK,\ast)\cong \scK$.
  \item If $F:\scA\times \scK^{\op}\to \Cat$ and $G:\scK\times \scB^{op}\to \Cat$ are functors, we have an induced functor
  $$\scA\times \scB^{op}\to \Cat,\quad (A,B)\mapsto C(F(A,-),\scK,G(-,B)$$
  \item Given functors $F$ and $G$ as in (3), and functors $U:\scA^{\op}\to \Cat$ and $V:\scB\to \Cat$, then
  $$U\otimes_{\scA}C(F,\scK,G)\otimes_{\scB} V\cong C(U\otimes_{\scA}F,\scK,G\otimes_{\scB} V).$$
  \item Given functors $F:\scK^{\op}\to \Cat,\ G:\scK\times \scL^{\op}\to \Cat$ and $H:\scL\to \Cat$, then
  $$C(C(F,\scK,G),\scL,H)\cong C(F,\scK,C(G,\scL,H)).$$
  \item  There is a natural transformation $\varepsilon: C(\underline{\scK},\scK,G)\Rightarrow G$, defined by
  $$\varepsilon(k):C(\scK(-,k,\scK,G)\to G(k), \quad (\beta,k_0,y)\mapsto G(\beta)(y),$$
  which is a weak equivalence. Dually, there is a natural transformation $C(F,\scK,\underline{\scK})\to F$ which is a    
  weak equivalence.
 \end{enumerate}
\end{leer}

\begin{proof} 
 (1),...,(5) follow by inspection of the definitions. For (6) note, that $\varepsilon(k)$ has a section
 $$s_k:G(k)\to C(\scK(-,k,\scK,G),\quad y\mapsto (\id_k,k,y),$$
 and there is a natural transformation 
 $$\tau: \Id_{C(\scK(-,k,\scK,G)}\Rightarrow s_k\circ \varepsilon(k),\quad (\beta,k_0,y)\xrightarrow{(\id,\beta,\id)} (\id_k,k,G(\beta)(y))$$
 so that $B(\varepsilon(k))$ is a homotopy equivalence.
\end{proof}


\begin{prop}\label{6_3}
\begin{enumerate}
 \item For $i=0,1,2,3$ the resolved cosimplicial categories $C(\Delta,\Delta,D_i)$ define good categoric realization functors.
 \item For $i=0,1,2,3$ the functor 
 $$\SCat\to \Cat,\quad \scC_* \mapsto F_{D_i}(C(\scC_*,\Delta,\Delta))$$
 is good.
\end{enumerate}
\end{prop}
\begin{proof}
\begin{eqnarray*}
F_{C(\Delta,\Delta,D_i)}(\scC_*) &= & \scC_*\otimes_\Delta C(\Delta,\Delta,D_i)\\
&=& C(\scC_*,\Delta,D_i)\\
&\simeq & C(\scC_*,\Delta,\ast)\\
&=& \scC_*\int \Delta.
\end{eqnarray*}
\begin{eqnarray*}
F_{D_i}(C(\scC_*,\Delta,\Delta))&=& C(\scC_*,\Delta,\Delta)\otimes_{\Delta} D_i\\
&=& C(\scC_*,\Delta,D_i)\\
&\simeq & \scC_*\int \Delta.
\end{eqnarray*}
where $\simeq$ stands for weakly equivalent.
According to Thomason \cite{T}, we have 
$$ B\left(\scC_*\int \Delta\right)\simeq B\left(B(\scC_*),\Delta,\ast\right)\cong B\left(\ast,\Delta^{\op}, B(\scC_*)\right).$$
Since $B(\scC_*)$ is a proper simplicial space, i.e. the inclusions $sB(\scC_n)\subset B(\scC_n)$
of the degenerate elements are cofibrations, the homotopy colimit $B\left(\ast,\Delta^{\op}, B(\scC_*)\right)$
is homotopy equivalent to the topological realization $|B(\scC_*)|$.
\end{proof}
Unfortunately, the resolutions we chose are not product preserving.

\section{Algebras}\label{sec.alg}

In this section we will rely heavily on \cite{FSV2} and we use its notation.

Let $\scO$ be a $\Sigma$-free operad in $\Cat$, let $\scO\mbox{-}\Cat$ be its category of algebras, and $\hscO$ its associated
category of operators. Then $\hscO$ is the operad $\scO$, considered as a symmetric monoidal category, with the projections added
(for details see \cite[Sect. 2]{FSV2}).

If $X_\ast$ is a simplicial $\scO$-algebra, $C(X_\ast,\Delta,\Delta)$ ceases to be a simplicial $\scO$-algebra, but in each degree $k$ it defines an
$\hscO$-diagram, which we, in abuse of notation, denote by 
$$
C(X_\ast,\Delta,\Delta(k,-)):\hscO\to \Cat,\quad n\mapsto C(X^n_\ast,\Delta,\Delta(k,-)).
$$
Since there is a weak equivalence $C(X_\ast,\Delta,\Delta)\to X_\ast$ it is easy to check that this is a special $\hscO$-diagram, i.e.
the $n$ projections define a weak equivalence
$$C(X^n_\ast,\Delta,\Delta(k,-)) \to \left(C(X_\ast,\Delta,\Delta(k,-))\right)^n.$$

In \cite{FSV2} we constructed a rectification functor
$$R: \Cat^{\hscO}\to \scO\mbox{-}\Cat$$
with nice properties, where $\Cat^{\hscO}$ denotes the category of $\hscO$-diagrams. Applying $R$
to the $\hscO$-diagrams $C(X_\ast,\Delta,\Delta(k,-))$ we obtain a functor
$$ Q: \scS \scO\mbox{-}\Cat\to \scS \scO\mbox{-}\Cat,\quad  X_\ast\mapsto \left( [k]\mapsto R(C(X_\ast,\Delta,\Delta(k,-))\right),$$
where $\scS \scO\mbox{-}\Cat$ denotes the category of simplicial $\scO$ algebras in $\Cat$.

\textit{Claim:} $Q(X_\ast)$ is a good resolution of $X_\ast$ in the category $\scS \scO\mbox{-}\Cat$.

Before we prove this, we have to give a short recollection of the definition of $R$. Let $\tmT$ denote the groupoid of
planar trees and non-planar isomorphisms. Let $\scT$ be the category whose objects are isomorphism classes $[T]$ of trees $T\in\tmT$
and whose morphisms are generated by shrinking an internal edge or chopping off a subtree above an internal edge (for more details 
see \cite[Sect. 5]{FSV2}). There is a functor
$\underline{\scO}: \tmT^{\op} \to \Cat$ known from the construction of free operads (e.g. see \cite[Sect. 5.8]{BM}). 
Given an $\hscO$-diagram $G:\hscO \to \Cat$ there is also a functor $\lambda_G:\tmT \to \Cat$, defined as follows. 
Let $\Theta_n$ denote the tree with exactly one node and $n$ inputs. Any tree $T$ with a root node of valence $n$ decomposes 
uniquely into $n$ trees $T_1,\ldots,T_n$ whose outputs are grafted onto the inputs of $\Theta_n$. We denote 
this grafting operation by
$$
T=\Theta_n\circ(T_1\oplus\ldots\oplus T_n).
$$
We define $\lambda_G(\Theta_n)=G(1)^n$ and 
$$
\lambda_G(\Theta_n\circ(T_1\oplus\ldots\oplus T_n))=G(\In (T_1))\times\ldots \times G(\In (T_n))$$
where $\In (T_i)$ is the number of inputs of $T_i$.

 We define a diagram 
$$F^G: \scT\to \Cat, \quad [T]\mapsto \underline{\scO}\otimes_{[T]} \lambda_G$$
where the coend is taken over all representatives $T$ of $[T]$. The functor $R$ is given by
$$R: \Cat^{\hscO}\to \scO\mbox{-}\Cat,\quad G\mapsto C(\ast, \scT, F^G)=\scT\int F^G.$$
If $X$ is an $\scO$-algebra, we denote its associated $\hscO$-diagram by $\widehat{X}:\hscO\to \Cat$.
\begin{prop}\label{7_1}
 There is a natural weak equivalence $\zeta: Q(X_\ast) \to X_\ast$ of simplicial $\scO$-algebras.
\end{prop}
\begin{proof}
 We have weak equivalences
 $$
Q(X_\ast)= C\left(\ast, \scT, F^{C(X_\ast,\Delta,\Delta)}\right) \xrightarrow{\sim} C(\ast, \scT, F^{\widehat{X}_\ast})
 \xrightarrow{\sim} X_\ast.$$
 The first map is induced by the weak equivalence $\varepsilon: C(X_\ast,\Delta,\Delta)\to X_\ast$. It is a
 homomorphism and a weak equivalence by
 \cite[6.7]{FSV2}. The second map is a weak equivalence by \cite[7.1]{FSV2}.
\end{proof}

\begin{theo}\label{7_2}
For $D=D_1,\ D_3,\ D_4$ the functor
 $$\scS\Cat \to \Top, \quad X_\ast\mapsto B(Q(X_\ast)\otimes_\Delta D)$$
 is equivalent to the functor  $|-|\circ B^{\Delta^{\op}}$ via a chain of weak equivalences of homomorphisms of $B\scO$-algebras.
\end{theo}
\begin{proof}
 $$Q(X_\ast)\otimes_\Delta D=C\left(\ast, \scT,F^{C(X_\ast,\Delta,\Delta)}\right)\otimes_\Delta D\cong C\left(\ast, \scT,F^{C(X_\ast,\Delta,\Delta)}\otimes_\Delta D\right).$$
 Since $-\otimes_\Delta D$ is product preserving, this is an isomorphism of $\scO$-algebras. Now 
 $$ F^{C(X_\ast,\Delta,\Delta)}([T])\otimes_\Delta D=\underline{\scO}\otimes_{[T]} \lambda_{C(X_\ast,\Delta,\Delta)}\otimes_\Delta D
 \cong  F^{C(X_\ast,\Delta,\Delta)\otimes_\Delta D}([T])$$
 by inspection of the definition of $\lambda_G$.\\
 By \cite[6.5]{FSV2} there is a weak equivalence of $B\scO$-spaces
 $$B(\ast,\scT, B(F^{C(X_\ast,\Delta,\Delta)\otimes_\Delta D}))\to BC(\ast,\scT,  F^{C(X_\ast,\Delta,\Delta)\otimes_\Delta D })\cong B(Q(X_\ast)\otimes_\Delta D).$$
 By \cite[6.6]{FSV2} there is an isomorphism of $B\scO$-spaces
 $$ B\left(\ast,\scT, B(F^{C(X_\ast,\Delta,\Delta)\otimes_\Delta D})\right)\cong B(\ast,\scT,F^{B(C(X_\ast,\Delta,\Delta)\otimes_\Delta D)}).$$
 By the proof of \ref{6_3} there is a chain of weak equivalences
 $$\xymatrix{
 B(C(X_\ast, \Delta, \Delta)\otimes_\Delta D)\ar[r]^(.6)\cong & BC(X_\ast, \Delta, D)\ar[r] &BC(X_\ast, \Delta,\ast) \\
 & & B(\ast, \Delta^{\op}, B(X_\ast)) \ar[u]\ar[r] & |B(X_\ast)|
 }
 $$
 It follows that these weak quivalences extend to weak equivalences of $\widehat{B\scO}$-diagrams, which by \cite[6.7]{FSV2} induce weak equivalences
 of $B\scO$-spaces between
 $$B\left(\ast,\scT, B(F^{C(X_\ast,\Delta,\Delta)\otimes_\Delta D})\right)\quad \textrm{and} \quad B\left(\ast,\scT, F^{\widehat{|B(X_\ast)|}}\right).$$
 Here note that $|B(X_\ast)|$ is a $B\scO$-algebra. By \cite[7.1]{FSV2} there is a weak equivalence of $B\scO$-algebras
 $$ B(\ast,\scT, F^{\widehat{|B(X_\ast)|}})\to |B(X_\ast)|$$
 which completes the proof.
\end{proof}


%
%
%
%

\end{document}